\documentclass[preprint,12pt]{elsarticle1}

\usepackage{subfig}
\usepackage{graphicx}
\usepackage{caption,color}   
\usepackage{amssymb}
\usepackage{amsthm}
\usepackage{amsmath}
\usepackage{epic}
\usepackage{setspace}
\usepackage{float}
\usepackage{multirow}

\usepackage{hyperref}
\usepackage{xcolor}
\usepackage{marginnote}

\newtheorem{thm}{Theorem}[section]
\newtheorem{cor}[thm]{Corollary}

\newtheorem{lem}[thm]{Lemma}
\newtheorem{pro}[thm]{Proposition}
\newtheorem{defi}[thm]{Definition}

\numberwithin{equation}{section}
\journal{}

\begin{document}
\begin{spacing}{1.15}
\begin{frontmatter}
\title{\textbf{The algebraic multiplicity of the spectral radius of a hypertree}}

\author{Lixiang Chen}\ead{chenlixiang@hrbeu.edu.cn}
\author{Changjiang Bu}\ead{buchangjiang@hrbeu.edu.cn}

\address{College of Mathematical Sciences, Harbin Engineering University, Harbin, PR China}

\begin{abstract}
It is well-known that the spectral radius of a connected uniform hypergraph is an eigenvalue of the hypergraph.
However, its algebraic multiplicity remains unknown.
In this paper, we use the Poisson Formula and matching polynomials to determine the algebraic multiplicity of the spectral radius of a uniform hypertree.
\end{abstract}


\end{frontmatter}

\section{Introduction}
From the Perron-Frobenius Theorem (for matrices), it is known that the spectral radius of a connected graph is an eigenvalue of the graph with the algebraic multiplicity 1.
Part of the Perron-Frobenius Theorem has been generalized to tensors, in particular, it is known that the spectral radius of a connected uniform hypergraph is an eigenvalue \cite{chang2008perron}. However, it is unknown what its algebraic multiplicity is.
In this paper, we aim to determine the algebraic multiplicity of the spectral radius of a uniform hypertree.

The characteristic polynomial of a hypergraph is defined to be the characteristic polynomial of its adjacency tensor.
The Poisson Formula, given in \cite[Chapter 3, Theorem 3.4]{Cox2005using}, is a useful method for computing the characteristic polynomials of hypergraphs, particularly hypertrees \cite{bao2021com,chen2021reduction,cooper2015Computing}.
Cooper and Dutle \cite{cooper2015Computing} gave the characteristic polynomial of the (so-called) ``all-one" tensors and the $3$-uniform hyperstar.
Bao et al. \cite{bao2021com} provided a combinatorial method for computing  the characteristic polynomials of hypergraphs with cut vertices, and gave the characteristic polynomial of a $k$-uniform hyperstar.
The authors gave a reduction formula for the characteristic polynomial of $k$-uniform hypergraphs with pendant edges in \cite{chen2021reduction}.
Besides, they used the reduction formula to derive the characteristic polynomial and all distinct eigenvalues of $k$-uniform loose hyperpaths \cite{chen2021reduction}.

The matching polynomial of a tree coincides with its characteristic polynomial, as established in Corollary 2.1 of \cite{Godsil1978onthematching}.
However, this correlation does not extend to $k$-uniform hypertrees with $k\geq3$.
Zhang et al. showed that the set of roots of the matching polynomial of a $k$-uniform hypertree is a sub-set of its spectrum \cite{zhang2017tree}.
Based on Zhang et al.'s results, Clark and Cooper determined all eigenvalues (without multiplicity) of a $k$-uniform hypertree $T$ by roots of the matching polynomials of all sub-hypertrees of $T$ \cite{clark2018hypertree}.
Su et al. use  the matching polynomials to investigate a perturbation on the spectral radius of hypertrees \cite{Su2018matching}.


The rest of this paper is organized as follows: In Section \ref{sec2}, we present some notation and lemmas about the Poisson Formula for resultants (Section \ref{sec2.1}), the characteristic polynomial of a hypergraph (Section \ref{sec2.2}), and the matching polynomial of a hypergraph (Section \ref{sec2.3}).
In Section \ref{sec3}, we apply the Poisson Formula to determine the algebraic multiplicity of the spectral radius of a $k$-uniform hypertree.

\section{Preliminaries}\label{sec2}
In this section, we present some basic notation and  auxiliary lemmas regarding the Poisson Formula for resultants, the characteristic polynomial and matching polynomial of a hypergraph.
\subsection{Resultants}\label{sec2.1}


For a positive integer $n$, let $\left[ n \right] = \left\{ {1, \ldots ,n} \right\}$.
Let $F_1,F_2,\ldots,F_n$ be homogeneous polynomials over an algebraically closed field $\mathbb{K}$ in
variables $x_1, . . . , x_n$, where the degree of $F_i$ is $d_i$ for $i \in [n]$.
Denote
\begin{eqnarray*}
  \overline{F_i}&=&\overline{F_i}(x_1,x_2,\ldots,x_{n-1})=F_i(x_1,x_2,\ldots,x_{n-1},0) \\
  f_i &=& f_i(x_1,x_2,\ldots,x_{n-1})=F_i(x_1,x_2,\ldots,x_{n-1},1).
\end{eqnarray*}
Observe that $\overline{F_i}$ are still homogeneous, but  $f_i$ are not homogeneous in general for $i \in [n]$.

Let $I =\langle f_1,\ldots,f_{n-1}\rangle \subset \mathbb{K}[x_1,\ldots,x_n]$ be the ideal generated by $f_i$ for $i=1,\ldots,n-1$.
It implies that the set of solutions of the system $f_1=f_2=\cdots=f_{n-1}=0$ are the variety $\mathcal{V}(I)$.
Given a polynomial  $f \in \mathbb{K}[x_1,\ldots,x_n]$, define a linear map $m_f$ from $\mathbb{K}[x_1,\ldots,x_n]/ I$ to itself using the multiplication.
More precisely, the polynomial $f$ gives the coset $[f] \in \mathbb{K}[x_1,\ldots,x_n]/ I $, and the linear map $m_f$ is defined by the rules: if $[g] \in \mathbb{K}[x_1,\ldots,x_n]/ I$, then
$$m_f([g])=[f]\cdot [g]=[fg]\in \mathbb{K}[x_1,\ldots,x_n]/ I. $$
The ensuing statement of the Poisson Formula for resultants follows from \cite[Chapter 3, Theorem 3.4]{Cox2005using}, which is different from the original one in \cite[Proposition 2.7]{Jouanolou1991Le}.

\begin{lem}\cite[Poisson Formula for resultants]{Cox2005using}
If $\mathrm{Res}(\overline{F_1},\ldots,\overline{F_{n-1}}) \neq 0$, then the quotient ring $A=\mathbb{K}[x_1,\ldots,x_n]/ \langle f_1,\ldots,f_{n-1}\rangle $ has dimension $d_1\cdots d_{n-1}$ as a vector space over $\mathbb{K}$, and
\begin{equation}\label{eq2.1}
\mathrm{Res}(F_1,\ldots,F_n)=\mathrm{Res}(\overline{F_1},\ldots,\overline{F_{n-1}})^{d_n}\det(m_{f_n}:A\rightarrow A),
\end{equation}
where $m_{f_n}:A\rightarrow A$ is the linear map given by multiplication by $f_n$.
\end{lem}

There is a one-one correspondence between the eigenvalues of the linear map $m_{f} : \mathbb{K}[x_1,\ldots,x_n]/ I \rightarrow \mathbb{K}[x_1,\ldots,x_n]/ I$ and the points of the variety  $\mathcal{V}(I)$,
so the characteristic polynomial of the linear map $m_{f}$ can be expressed in terms of the points of $\mathcal{V}(I)$.

\begin{pro}\cite[Chapter 4, Proposition 2.7]{Cox2005using}\label{pro2.2}
Let $\mathbb{K}$ be an  algebraically closed field and let $I$ be a zero-dimensional ideal in $\mathbb{K}[x_1,\ldots,x_n]$. If $f \in \mathbb{K}[x_1,\ldots,x_n]$, then
$$\det(\lambda I-m_f)=\prod_{\mathbf{p} \in \mathcal{V}(I)}\left(\lambda-f(\mathbf{p})\right)^{m(\mathbf{p})},$$
where $m(\mathbf{p})$ is the multiplicity \footnote{The multiplicity is sometimes called the local intersection multiplicity, and its definition can be found in \cite[Chapter 4, \S 2]{Cox2005using}, but it is inconsequential in this paper.} of the point $\mathbf{p} \in \mathcal{V}(I)$.
\end{pro}
By Proposition \ref{pro2.2}, it implies that \eqref{eq2.1} can be rewritten as
\begin{equation}\label{eq2.2}
\mathrm{Res}(F_1,\ldots,F_n)=\mathrm{Res}(\overline{F_1},\ldots,\overline{F_{n-1}})^{d_n}\prod_{\mathbf{p} \in \mathcal{V}}f_n(\mathbf{p})^{m(\mathbf{p})},
\end{equation}
where $\mathcal{V}$ is the affine variety defined by the polynomials $f_i$ for all $i \in [n-1]$.
The formula \eqref{eq2.2} is also named Poisson Formula for resultants in the monograph \cite[Chapter 13, Theorem 1.3]{Gel994Res}.

\subsection{The characteristic polynomial of a hypergraph}\label{sec2.2}


A hypergraph $H=(V,E)$ is called \emph{$k$-uniform} if each edge of $H$ contains exactly $k$ vertices.
Similar to the relation between graphs and matrices, there is a natural correspondence between uniform hypergraphs and tensors.
For a $k$-uniform hypergraph $H$ with $n$ vertices, its \emph{(normalized) adjacency tensor} ${A}_H=(a_{i_1i_2\ldots i_k})$ is a $k$-order $n$-dimensional tensor \cite{cooper2012spectra}, where
\[{a_{{i_1}{i_2} \ldots {i_k}}} = \left\{ \begin{array}{l}
 \frac{1}{{\left( {k - 1} \right)!}},{\kern 37pt}\mathrm{ if}{\kern 2pt}{ \left\{ {{i_1},{i_2},\ldots ,{i_k}} \right\} \in {E}}, \\
 0, {\kern 57pt}\mathrm{ otherwise}. \\
 \end{array} \right.\]
When $k=2$, ${A}_H$ is the usual adjacency matrix of the graph $H$.

Let $H=(V,E)$ be a $k$-uniform hypergraph with $V=[n]$.
Given a hyperedge $e\in E$, and a vector $\mathbf{x} =(x_1,x_2,\ldots,x_n)^{\top}\in {\mathbb{C}^n}$, let $\mathbf{x}_{e}=\prod\limits_{v \in e}x_v$.
Let $E_v=\{e\in E: v\in e\}$ denote the set of hyperedges containing the vertex $v$.
If there exists a nonzero vector $\mathbf{x}$ such that for each $i \in [n]$,
\[\lambda x_i^{k-1} = \sum\limits_{{i_2}, \ldots ,{i_k}=1}^n {{a_{i{i_2} \cdots {i_k}}}{x_{{i_2}}} \cdots {x_{{i_k}}}} ,\]
or equivalently, for each $v \in V$,
\[\lambda x_v^{k-1} = \sum\limits_{e \in E_v} \mathbf{x}_{e \setminus \{v\}}, \]
then $\lambda$ is called an \emph{eigenvalue} of $H$ and $\mathbf{x}$ is an \emph{eigenvector} of $H$ corresponding to $\lambda$ \cite{lim2005singular,qi2005eigenvalues}.

For each $v \in V$, define
\begin{equation*}
  F_v=F_v(x_1,x_2,\ldots,x_{n})=\lambda x^{k-1}_v-\sum_{e \in E_v}\mathbf{x}_{e \setminus \{v\}}.
\end{equation*}
The polynomial
\begin{equation*}
  \phi_H(\lambda)\equiv \mathrm{Res}(F_v: v \in V)
\end{equation*}
in the indeterminant $\lambda$ is called the \emph{characteristic polynomial} of $H$.
For a fixed vertex $u \in V$, let
\begin{equation*}
  f_v=F_v|_{x_u=1}=F_v(x_1,\cdots,x_{u-1},1,x_{u+1},\ldots,x_{n}).
\end{equation*}
Let $\mathcal{V}$ be the affine variety defined by the polynomials $f_v$ for all $ v \in V \setminus \{u\}$.
Let $H-u$ denote the hypergraph obtained from $H$ by removing the vertex $u$ and all hyperedges incident to $u$.
Applying the Poisson Formula to the characteristic polynomial $\phi_H(\lambda)$, a reduction formula for $\phi_H(\lambda)$ is derived as follows.

\begin{lem}\cite[Formula (1)]{chen2021reduction}\label{lem2.3}
Let $H=(V,E)$ be a $k$-uniform hypergraph with the vertex $u$.
Then the  characteristic polynomial
\begin{align*}
 \phi_H(\lambda)=\phi_{H-u}(\lambda)^{k-1}\prod_{\mathbf{p} \in \mathcal{V}}(\lambda-\sum_{e \in E_u}\mathbf{p}_{e\setminus \{u\} })^{m(\mathbf{p})},
 \end{align*}
where $m(\mathbf{p})$ is the multiplicity of $\mathbf{p}$ in $\mathcal{V}$.
\end{lem}

\begin{defi}[\emph{Cut vertices} of hypergraphs]\cite{bao2021com}\label{dificutvertex}
Let $k \geq 3$, and let $H = (V, E)$ be a $k$-uniform connected hypergraph and $u \in V$.
Denote $E_{\widetilde{u}}=\{e \setminus \{u\}: e \in E_u \}$, and note that the hyperedge $\widetilde{e} =e \setminus \{u\} \in E_{\widetilde{u}}$ has $k-1$ vertices.
Deleting the vertex $u$ and changing $e$ to $\widetilde{e}$ for every $e \in E_u$, it can get a non-uniform hypergraph $\widetilde{H}=(\widetilde{V},\widetilde{E})$ with $\widetilde{V}=V\setminus {u}$ and $\widetilde{E}=(E \setminus E_u)\cup E_{\widetilde{u}} $.
The vertex $u$ is called a cut vertex  if $\widetilde{H}$ is not connected.
\end{defi}

The vertex $u$ may not necessarily be a cut vertex of a connected hypergraph $H$, even if $H-u$ is not connected. For instance, when $k \geq 3$, the vertex $u$ with degree one of a $k$-uniform hypertree $T$ is not a cut vertex, even if $T-u$ is not connected.
Suppose that $\widetilde{H}_1=(\widetilde{V}_1,\widetilde{E}_1),\ldots,\widetilde{H}_n=(\widetilde{V}_n,\widetilde{E}_n)$ are  connected components of
$\widetilde{H}$.
For each $i \in [n]$, denote the induced sub-hypergraph of $H$ on $\widetilde{V}_i \cup \{u\}$ by $\widehat{H}_i$,
then we call $\widehat{H}_i$ a \emph{branch} of $H$ associated with $u$.
It implies that $H$ can be obtained by coalescing $\widehat{H}_1, \ldots,\widehat{H}_n$ to the vertex $u$.
Recall that the variety $\mathcal{V}(H)$ is defined by the polynomials
$$f_v=\lambda x^{k-1}_v-\sum_{e \in E_v(H)}\mathbf{x}_{e \setminus \{v\}}|_{x_u=1}$$ for all $ v \in V \setminus \{u\}$.
In \cite{bao2021com}, it is shown that if $u$ is a cut vertex of $H$, then $$\mathcal{V}(H)=\bigoplus_{i=1}^n\mathcal{V}(\widehat{H}_i).$$
Bao et al. subsequently gave a reduction formula for the characteristic polynomials of hypergraphs with cut vertices in terms of the linear map \cite[Corollary 3.2]{bao2021com}.
For the convenience of use in Section \ref{sec3}, we restate their formula following a similar approach to the reformulation of \eqref{eq2.1} as \eqref{eq2.2}.

\begin{lem}\cite[Corollary 3.2]{bao2021com}\label{lembao}
Let $H$ be a $k$-uniform hypergraph with a cut vertex $u$ and branches $\widehat{H}_1, \cdots,\widehat{H}_n$.
Denote $\mathcal{V}^{(i)}=\mathcal{V}(\widehat{H}_i)$ and $E^{(i)}_u=E_u(\widehat{H}_i) \bigcap E_u(H)$.
Then
$$\phi_H(\lambda)=\phi_{H-u}(\lambda)^{k-1}\prod_{\mathbf{p}^{(i)} \in \mathcal{V}^{(i)} \atop i\in [n]}(\lambda-\sum_{e \in E^{(i)}_u \atop i\in [n] }\mathbf{p}^{(i)}_{e \setminus \{u\}})^{\prod_{i=1}^{n}m(\mathbf{p}^{(i)})},$$
where $m(\mathbf{p}^{(i)})$ is the multiplicity of $\mathbf{p}^{(i)}$ in $\mathcal{V}^{(i)}$ for each $i \in [n]$.
\end{lem}

When one of the branches is the one-edge hypergraph, it implies that $H$ has a pendant edge incident to $u$.
A more explicit reduction formula for hypergraphs with pendant edges is shown as follows.

\begin{lem}\cite[Theorem 3.2]{chen2021reduction}\label{lem2.6}
Let $H$ be a $k$-uniform hypergraph with a pendant edge incident to the non-pendent vertex $u$, and we define $\widehat{H}$ as the $k$-uniform hypergraph obtained by removing the pendant edge and pendent vertices on it from $H$. Then
\begin{align*}
&\phi_H(\lambda)=\\
&\phi_{H-u}(\lambda)^{k-1}\prod_{\mathbf{p} \in \mathcal{V}(\widehat{H})}(\lambda-\sum_{e\in E_u(\widehat{H})}\mathbf{p}_{e\setminus \{u\}})^{m(\mathbf{p})K_1}\prod_{\mathbf{p} \in \mathcal{V}(\widehat{H})}(\lambda-\frac{1}{\lambda^{k-1}}-\sum_{e\in E_u(\widehat{H})}\mathbf{p}_{e\setminus \{u\}})^{m(\mathbf{p})K_2},
\end{align*}
where $K_1=(k-1)^{k-1}-k^{k-2}$ and $K_2=k^{k-2}$.
\end{lem}

\subsection{The matching polynomial of a hypergraph}\label{sec2.3}

The \emph{matching} of a $k$-uniform hypergraph $H=(V,E)$ is a set of the pairwise non-adjacent edges in $E$.
The $t$-matching is a matching consisting of $t$ edges and the number of $t$-matching of $H$ is denoted by $m_t(H)$.
Set $m_0(H)=1$.
The \emph{matching polynomial} of $H$  is defined in \cite{Su2018matching} as
$$\varphi_H(\lambda)=\sum_{t\geq 0}(-1)^t m_t(H)x^{|V|-tk}.$$

Some classical results on the matching polynomials of  a graph are extended to the hypergraph case as follows.

\begin{lem}\cite{Su2018matching}\label{lemmatching}
Let $H=(V,E)$ and $G$ be two $k$-uniform hypergraphs. We use $H+G$ to denote the disjoint union of $H$ and $G$, and use $H-e$ to
denote the induced sub-hypergraph of $H$ on the $V\setminus e$ for $e\in E$. Then
\begin{enumerate}[(1)]
\item \label{lem2.7(1)}  $\varphi_{H+G}(\lambda)=\varphi_{H}(\lambda) \varphi_{G}(\lambda)$.
\item \label{lem2.7(2)}  $\varphi_{H}(\lambda) = \lambda\varphi_{H-u}(\lambda)-\sum_{e \in E_u}\varphi_{H-e}(\lambda)$.
\end{enumerate}
\end{lem}

Zhang et al. \cite{zhang2017tree} showed that the set of roots of the matching polynomial of a $k$-uniform hypertree is a sub-set of its spectrum.
And this result was extended by Clark and Cooper \cite{clark2018hypertree} as follows.

\begin{lem}\cite[Theoerm 2]{clark2018hypertree}\label{lem2.8}\footnote{The definition of the matching polynomial of a hypergraph varies between \cite{clark2018hypertree,zhang2017tree} and \cite{Su2018matching}. However, Lemma \ref{lem2.8} is applicable to both definitions as shown in \cite{Su2018matching}.}
Let $T$ be a $k$-uniform hypertree for $k \geq 3$.
Then $\lambda$ is an eigenvalue of $T$ if and only if there exists an induced sub-tree $\widehat{T}$ of $T$ such that $\lambda$ is a root of the matching polynomial $\varphi_{\widehat{T}}(\lambda)$.
\end{lem}

\begin{lem}[Proposition 8]\cite{Su2018matching}\label{lem2.9}
Let $k \geq 3$, and let $T$ be a $k$-uniform hypertree. Then the spectral radius $\rho(T)$ of $T$ is a simple root of the matching polynomial $\varphi_{T}(\lambda)$.
\end{lem}

\section{Main results}\label{sec3}
The algebraic multiplicity of the spectral radius of a $k$-uniform hypertree $T$ is determined in this section.


For a hypergraph $H=(V,E)$ with the vertex $u$,
recall that $F_v=\lambda x^{k-1}_v-\sum_{e \in E_v}\mathbf{x}_{e \setminus \{v\}}$ and $f_v=F_v|_{x_u=1}$ for all $v \in V$.
Let $\mathcal{V}$ be the affine variety defined by the polynomials $f_v$ for all $ v \in V \setminus \{u\}$.

\begin{lem}\label{lem3.1}
Let $T=(V,E)$ be a uniform hypertree with the vertex $u$, and let $\mathbf{p}\in \mathcal{V}$ have all coordinates nonzero. Then
\begin{equation*}
\mathbf{p}_{e\setminus \{u\} } = \frac{\varphi_{T-e}(\lambda)}{\varphi_{T-u}(\lambda)}
\end{equation*}for each $e \in E_u$.
Moreover, we have
$$\lambda-\sum_{e \in E_u}\mathbf{p}_{e \setminus \{u\}}=\frac{\varphi_{T}(\lambda)}{\varphi_{T-u}(\lambda)}.$$
\end{lem}

\begin{proof}
We prove the result by the induction on $|E|$.

When $|E|=1$. It is shown that $ \mathbf{p}_{e\setminus \{u\} }=\frac{1}{\lambda^{k-1}}$ in the \cite[Equation (5)]{chen2021reduction}, which implies that $$\mathbf{p}_{e\setminus \{u\} }=\frac{\varphi_{T-e}(\lambda)}{\varphi_{T-u}(\lambda)},$$ so the assertion holds.

Assuming the statement holds for any $|E|\leq m$, we consider the case $|E|=m+1$.

When $u$ is a cut vertex of $T$.
Note that there are $d_u$ ($> 1$) branches of the hypertree $T$ associated with $u$, and each $e \in E_u$ belongs to a distinct branch.
Suppose that $\widehat{T}_i $ are branches of $T$ with $e_i \in E_u$ for every $i \in [d_u]$.
For $\mathbf{q}^{(i)} \in \mathcal{V}(\widehat{T}_i)$ having all coordinates nonzero, we have
\begin{align}\label{eqx}
\mathbf{q}_{e_i \setminus \{u\}}^{(i)}&=\frac{\varphi_{\widehat{T}_i-{e_i}}(\lambda)}{\varphi_{\widehat{T}_i-u}(\lambda)}\notag \\
&=\frac{\varphi_{\widehat{T}_i-{e_i}}(\lambda)\prod_{ j \in [d_u] \atop i\neq j}\varphi_{\widehat{T}_j-u}(\lambda)}{\prod_{ j \in [d_u] }\varphi_{\widehat{T}_j-u}(\lambda)}
\end{align}
by the induction hypothesis,
since $|E(\widehat{T}_i)| \leq m$.
Note that $T-e_i$ is the disjoint union of $\widehat{T}_i-{e_i}$ and ${\widehat{T}_j-u}$ for all $i\neq j$, and $T-u$ is the disjoint union of ${\widehat{T}_j-u}$ for all $j \in [d_u]$.
It implies that
\begin{align*}
{\varphi_{T-e_i}(\lambda)}&={\varphi_{\widehat{T}_i-{e_i}}(\lambda)\prod_{ j \in [d_u] \atop i\neq j}\varphi_{\widehat{T}_j-u}(\lambda)}
\end{align*}
and
\begin{align*}
{\varphi_{T-u}(\lambda)}&={\prod_{ j \in [d_u] }\varphi_{\widehat{T}_j-u}(\lambda)}
\end{align*}
from Lemma \ref{lemmatching} (\ref{lem2.7(1)}).
For $\mathbf{p} \in \mathcal{V}$, we get $$\mathbf{p}_{e_i\setminus \{u\} }=\mathbf{q}_{e_i \setminus \{u\}}^{(i)}=\frac{\varphi_{T-e_i}(\lambda)}{\varphi_{T-u}(\lambda)}$$ from \eqref{eqx} and Lemma \ref{lembao}.

When $u$ is not a cut vertex of $T$.
Note that the degree of $u$ is one and we set the hyperedge containing $u$ as $e_0=\{v_1,v_2,\ldots,v_{k-1},v_{k}=u\}$.
Let $T\setminus e_0$ denote the hypergraph with $V(T\setminus e_0)=V$ and $E(T\setminus e_0)=E \setminus \{e_0\}$.
It is natural that $T\setminus e_0$ has $k$ connected components and we use $\widetilde{T}_t=(\widetilde{V}_t,\widetilde{E}_t)$ to denote the connected component containing $v_t$  for each $t \in [k]$.

For $v \in V$, recall that $F_v=F_v(x_i : i \in V)=\lambda x^{k-1}_v-\sum_{e \in E_v(T)}\mathbf{x}_{e \setminus \{v\}}$ and $f_v=F_v|_{x_u=1}$.
For all $t\in [k-1]$ and any $v \in \widetilde{V}_t \setminus \{v_t\}$, note that $f_v=f_v(x_i : i \in \widetilde{V}_t )$ is homogeneous.
Let $\mathbf{p}=(\mathbf{p}_i) \in \mathcal{V}$ have all coordinates nonzero, then we have
\begin{equation}\label{eq3.1}
f_v(\mathbf{p})=f_v(\mathbf{p}_i: i \in \widetilde{V}_t)=f_v({\mathbf{p}_i}/{\mathbf{p}_{v_t}}: i \in \widetilde{V}_t)=0.
\end{equation}

Let $\widetilde{F}_v=\widetilde{F}_v(x_i:i\in \widetilde{V}_t)=\lambda x^{k-1}_v-\sum_{e \in E_v(\widetilde{T}_t)}\mathbf{x}_{e \setminus \{v\}}$ and $\widetilde{f}_v=\widetilde{F}_v|_{x_{v_t}=1}$.
For $v \in \widetilde{V}_t \setminus \{v_t\}$, note that $\widetilde{F}_v=f_v$.
Set $\mathbf{q}_i={\mathbf{p}_i}/{\mathbf{p}_{v_t}}$ for $v\in \widetilde{V}_t$ and note that $\mathbf{q}_{v_t}=1$.
By \eqref{eq3.1}, we have
$$\widetilde{f}_v(\mathbf{q}_i: i \in \widetilde{V}_t)=\widetilde{F}_v(\mathbf{q}_i: i \in \widetilde{V}_t)|_{\mathbf{q}_{v_t}=1}=f_v(\mathbf{q}_i: i \in \widetilde{V}_t)=0$$
for all $ v \in \widetilde{V}_t \setminus \{v_t\}$.
Let the vector $\mathbf{q}=(\mathbf{q}_i)$ for $i \in \widetilde{V}_t \setminus \{v_t\}$,
indeed $\mathbf{q}$ is a point in the variety $\mathcal{V}(\widetilde{T}_t)$ defined by the polynomials $\widetilde{f}_v$ for all $ v \in \widetilde{V}_t \setminus \{v_t\}$.
From the assumption, we have
\begin{equation}\label{eq3.2}
\mathbf{q}_{e \setminus \{v_t\}}=\frac{\mathbf{p}_{e \setminus \{v_t\}}}{\mathbf{p}^{k-1}_{v_t}}=\frac{\varphi_{\widetilde{T}_t-{e}}(\lambda)}{\varphi_{\widetilde{T}_t-v_t}(\lambda)}
\end{equation}
for each $e \in E_{v_t}(\widetilde{T}_t)$.
For $\mathbf{p} \in \mathcal{V}$ having non-zero coordinates, we have
$$\mathbf{p}_{e \setminus \{v_t\}}=\frac{\varphi_{\widetilde{T}_t-{e}}(\lambda)}{\varphi_{\widetilde{T}_t-v_t}(\lambda)}\mathbf{p}^{k-1}_{v_t}$$
by \eqref{eq3.2}.
Note that
\begin{align*}
f_{v_t}(\mathbf{p})=\lambda \mathbf{p}_{v_t}^{k-1}-\sum_{e \in E_{v_t}(\widetilde{T}_t)}\mathbf{p}_{e \setminus \{v_t\}}-\mathbf{p}_{e_0 \setminus \{v_t,u\}}=0
\end{align*}
for all $t \in [k-1]$.
Then we have
\begin{align*}
\mathbf{p}_{e_0 \setminus \{v_t,u\}}&=\left(\lambda-\sum_{e \in E_{v_t}(\widetilde{T}_t)}\frac{\varphi_{\widetilde{T}_t-{e}}(\lambda)}{\varphi_{\widetilde{T}_t-v_t}(\lambda)}\right)\mathbf{p}^{k-1}_{v_t}\\
&=\frac{{\varphi_{\widetilde{T}_t}(\lambda)}}{{\varphi_{\widetilde{T}_t-v_t}(\lambda)}}\mathbf{p}^{k-1}_{v_t}
\end{align*}
from Lemma \ref{lemmatching} (\ref{lem2.7(2)}).
Combining these equations for all $t \in [k-1]$, we get
$$\prod_{t=1}^{k-1}\mathbf{p}_{e_0 \setminus \{v_t,u\}}=\prod_{t=1}^{k-1}\frac{{\varphi_{\widetilde{T}_t}(\lambda)}}{{\varphi_{\widetilde{T}_t-v_t}(\lambda)}}\mathbf{p}^{k-1}_{v_t},$$
which implies that
$$\mathbf{p}_{e_0 \setminus \{u\}}=\prod_{t=1}^{k-1}\frac{{\varphi_{\widetilde{T}_t-v_t}(\lambda)}}{{\varphi_{\widetilde{T}_t}(\lambda)}}.$$
Note that $T-e_0$ is the disjoint union of ${\widetilde{T}_t-v_t}$ for all $t \in [k-1]$, and $T-u$ is  the disjoint union of ${\widetilde{T}_t}$ for all $t \in [k-1]$.
Then we have $\mathbf{p}_{e_0 \setminus \{u\}}=\frac{\varphi_{T-e_0}(\lambda)}{\varphi_{T-u}(\lambda)}$ from Lemma \ref{lemmatching} (\ref{lem2.7(1)}).
By the induction, we have
\begin{equation*}
\mathbf{p}_{e\setminus \{u\} } = \frac{\varphi_{T-e}(\lambda)}{\varphi_{T-u}(\lambda)}
\end{equation*} for $e \in E_u$.
Hence, we get $$\lambda-\sum_{e \in E_u(T)}\mathbf{p}_{e \setminus \{u\}}=\frac{\varphi_{T}(\lambda)}{\varphi_{T-u}(\lambda)}$$
from Lemma \ref{lemmatching} (\ref{lem2.7(2)}).
\end{proof}

The support of a vector $\mathbf{p} \in \mathcal{V}$, denoted by $\mathrm{supp}(\mathbf{p})$, is the set of all indices for non-zero coordinates of $\mathbf{p}$.
Let the $|\mathrm{supp}(\mathbf{p})|$-dimensional vector $\mathbf{p}^*$ denote the non-zero projection of $\mathbf{p}$, i.e., $\mathbf{p}^*$ is the vector constructed from all the non-zero coordinates of $\mathbf{p}$.
We use  $H[U]$ to denote the induced sub-hypergraph of a hypergraph $H=(V,E)$ on $U \subseteq V$.
It is seen that $\mathbf{p}^* \in \mathcal{V}(H[\mathrm{supp}(\mathbf{p}) \cup \{u\}])$ only has nonzero coordinates.
Applying Lemma \ref{lem3.1} to $\mathbf{p}^* $,  we can directly extend Lemma \ref{lem3.1} to the following broader form.

\begin{cor}\label{cor3.2}
Let $T=(V,E)$ be a hypertree with the vertex $u$.
For $\mathbf{p}\in \mathcal{V}$, let $T_{\mathbf{p}}=T[\mathrm{supp}(\mathbf{p}) \cup \{u\}]$.
Then we have
\begin{equation*}
\mathbf{p}_{e\setminus \{u\} } = \frac{\varphi_{T_{\mathbf{p}}-e}(\lambda)}{\varphi_{T_{\mathbf{p}}-u}(\lambda)}
\end{equation*}for each $e \in E_u(T_{\mathbf{p}})$.
Moreover, we have
$$\lambda-\sum_{e \in E_u}\mathbf{p}_{e \setminus \{u\}}=\frac{\varphi_{T_{\mathbf{p}}}(\lambda)}{\varphi_{T_{\mathbf{p}}-u}(\lambda)}.$$
\end{cor}

We are now ready to  determine the algebraic multiplicity of the spectral radius of a uniform hypertree.
\begin{thm}
The algebraic multiplicity of the spectral radius of a $k$-uniform hypertree with $m$ edges is $k^{m(k-2)}$.
\end{thm}

\begin{proof}
From the Poisson Formula for hypergraphs and Corollary \ref{cor3.2}, we have
\begin{align}\label{eq3.3}
 \phi_T(\lambda)=\phi_{T-u}(\lambda)^{k-1}\prod_{\mathbf{p} \in \mathcal{V}}\left(\frac{\varphi_{T_{\mathbf{p}}}(\lambda)}{\varphi_{T_{\mathbf{p}}-u}(\lambda)}\right)^{m(\mathbf{p})}
 \end{align}
for a hypertree $T=(V,E)$ and any $u \in V$.

Lemma \ref{lem2.8} shows that the roots of the matching polynomials of sub-hypertrees of a hypertree $T$ are eigenvalues of $T$.
For a connected hypergraph $G$ and its proper sub-graph $G'$, it is known that $\rho(G') < \rho(G)$ \cite[Corollary 3.5]{khan2015radius}.
It tells that $\rho(G)$ is not an eigenvalue of any proper sub-graph of $G$.
It implies that $\rho(T)$ is not a root of the matching polynomials of any proper sub-graph of $T$.
Then we know that the degree of the factor $\lambda-\rho(T)$ in \eqref{eq3.3} is solely determined by $\varphi_{T}(\lambda)$, and its degree in $\varphi_{T}(\lambda)$ is one, as confirmed by Lemma \ref{lem2.9}.
Let $\mathrm{am}(T)$ denote the algebraic multiplicity of the spectral radius $\rho(T)$, and let $\mathcal{V}^*=\{\mathbf{p} \in \mathcal{V}: T_{\mathbf{p}}=T\}$.
By \eqref{eq3.3}, we get
\begin{equation}\label{eq3.4}
 \mathrm{am}(T) =\sum_{\mathbf{p} \in \mathcal{V}^*}m(\mathbf{p}).
\end{equation}

Let $\widetilde{T}$ denote the hypertree obtained from $T$  adding a pendant edge at the vertex $u$.
From Lemma \ref{lem2.6}, we get
\begin{equation}\label{eq3.5}
  \phi_{\widetilde{T}}(\lambda)=\phi_{\widetilde{T}-u}(\lambda)^{k-1}\prod_{\mathbf{p} \in \mathcal{V}}\left(\frac{\varphi_{T_{\mathbf{p}}}(\lambda)}{\varphi_{T_{\mathbf{p}}-u}(\lambda)}\right)^{m(\mathbf{p})(k-1)^{k-1}-k^{k-2}}\prod_{\mathbf{p} \in \mathcal{V}}\left(\frac{\varphi_{\widetilde{T}_{\mathbf{p}}}(\lambda)}{\varphi_{\widetilde{T}_{\mathbf{p}}-u}(\lambda)}\right)^{m(\mathbf{p})k^{k-2}},
\end{equation}
where $\widetilde{T}_{\mathbf{p}}$ denotes the hypertree obtained from $T_{\mathbf{p}}$ by adding a pendant edge at the vertex $u$.
Employing a similar technique as used in obtaining \eqref{eq3.4} from \eqref{eq3.3},
we derive  the following equation from \eqref{eq3.5}:
\begin{equation*}
 \mathrm{am}(\widetilde{T}) =k^{k-2}\sum_{\mathbf{p} \in \mathcal{V}^*}m(\mathbf{p})=k^{k-2} \mathrm{am}(T),
\end{equation*}
which implies that the algebraic multiplicity of the spectral radius increases $k^{k-2}$-fold when a pendant edge is added to a hypertree.
By starting  with a hypertree having no edge, i.e., an isolated vertex, we obtain that the algebraic multiplicity of the spectral radius of a $k$-uniform hypertree with $m$ edges is $k^{m(k-2)}$.
\end{proof}

\noindent{\textbf{Remark}:} The $k$-power hypergraph $G^{(k)}$ is the $k$-uniform hypergraph that is obtained by adding $k-2$ new vertices to each edge of a graph $G=(V,E)$ for $k \geq 3$.  It is shown that the algebraic multiplicity of the spectral radius of $G^{(k)}$ is $k^{|E|(k-3)+|V|-1}$ in \cite{chen2023parity}.
When $G$ is a tree, $G^{(k)}$ becomes a $k$-uniform power hypertree. In this case, the result provided in \cite{chen2023parity} is consistent with the main result of this paper.

\section*{References}
\bibliographystyle{plain}
\bibliography{atbib}
\end{spacing}
\end{document}